\documentclass[]{interact}

\usepackage{epstopdf}
\usepackage[caption=false]{subfig}

\usepackage[numbers,sort&compress]{natbib}
\bibpunct[, ]{[}{]}{,}{n}{,}{,}
\makeatletter
\def\NAT@def@citea{\def\@citea{\NAT@separator}}
\makeatother

\theoremstyle{plain}
\newtheorem{theorem}{Theorem}[section]
\newtheorem{lemma}[theorem]{Lemma}
\newtheorem{corollary}[theorem]{Corollary}

\theoremstyle{definition}
\newtheorem{definition}[theorem]{Definition}

\theoremstyle{remark}
\newtheorem{remark}{Remark}

\newcommand{\beq}{\begin{equation}} \newcommand{\eeq}{\end{equation}}
\newcommand{\bea}{\begin{eqnarray}} \newcommand{\eea}{\end{eqnarray}}
\newcommand{\bear}{\begin{eqnarray*}} \newcommand{\eear}{\end{eqnarray*}}
\newcommand{\lb}{\label} 
\newcommand{\rf}[1]{(\ref{#1})}

\usepackage{hyperref}
\hypersetup{colorlinks=true, linkcolor=blue, citecolor=blue, filecolor=blue, urlcolor=blue,
            pdfproducer={Latex},
            pdfcreator={pdflatex}}

\begin{document}

\title {Noether-type theorem for fractional variational problems depending on fractional derivatives of functions}

\author{
\name{M.~J.~Lazo\textsuperscript{a}\thanks{CONTACT M.~J.~Lazo. Email: matheuslazo@furg.br}, G.~S.~F.~Frederico\textsuperscript{b} and P.~M.~Carvalho-Neto\textsuperscript{c}}
\affil{\textsuperscript{a}Federal University of Rio Grande, Rio Grande - RS, Brazil; \textsuperscript{b}Federal University of Cear\'a, Campus de Russas, Russas, Brazil; \textsuperscript{c}Federal University of Santa Catarina, Florian\'opolis - SC, Brazil.}
}

\maketitle

\begin{abstract}

In the present work, by taking advantage of a so-called practical limitation of fractional derivatives, namely, the absence of a simple chain and Leibniz's rules, we proposed a generalized fractional calculus of variation where the Lagrangian function depends on fractional derivatives of differentiable functions. The Euler-Lagrange equation we obtained generalizes previously results and enables us to construct simple Lagrangians for nonlinear systems. Furthermore, in our main result, we formulate a Noether-type theorem for these problems that provides us with a means to obtain conservative quantities for nonlinear systems. In order to illustrate the potential of the applications of our results, we obtain Lagrangians for some nonlinear chaotic dynamical systems, and we analyze the conservation laws related to time translations and internal symmetries.

\end{abstract}

\begin{keywords}
Fractional Noether-type theorem; fractional calculus of variation; nonlinear and  chaotic systems
\end{keywords}



\section{Introduction}

The calculus with non-integer order derivatives and integrals, historically called fractional calculus, originated in 1695 from a Leibniz letter in response to a question from l'H\^opital \cite{OldhamSpanier}. The fractional calculus caught the attention of important mathematicians throughout its history, like Euler, Fourier, Liouville, Riemann, and others. Despite being an active area of study in mathematics for more than three centuries, only recently it attained more attention from researchers from other areas due to the emergence of several important applications in various scientific fields (for examples, see\cite{SATM,Kilbas,Diethelm,Hilfer,Magin,SKM}).

Classical mechanics is an example where the fractional calculus has a remarkable application. In a seminal work, Riewe \cite{Riewe1,Riewe2} showed that the equation of motion of a particle under a frictional force can be obtained from a quadratic Lagrangian containing a half order derivative. The importance of this result lies in the fact that non-conservative forces are beyond the usual variational formulation of the action principle \cite{Bauer}, and hence, they are beyond the most advanced methods of Lagrangian mechanics. Most important, different from other approaches in the literature, the so-called fractional variational calculus introduced by Riewe give us a physically meaningful Lagrangian function \cite{Riewe1,Riewe2,LazoCesar}. Therefore, since most real-world systems are not conservative, the fractional calculus of variations provides us with a simple and more realistic approach to study a large variety of problems \cite{MalinowskaTorres}. Consequently, in the last few years, the fractional calculus of variations becomes a very active research area and several generalizations of Riewe's original approach with applications in diverse fields are proposed \cite{MalinowskaTorres}. Among these generalizations, recently one of us with a collaborator demonstrated that we can formulated a consistent action principle for linear nonconservative systems, correcting a mathematical inconsistency found in the Riewe's original approach, by considering Lagrangian functions depending on usual integer order and Caputo derivatives \cite{LazoCesar}.

The fractional calculus of variation with mixed integer and Caputo derivatives allows us to employ the mathematical tools of analytical mechanics to investigate nonconservative systems, as for example, the Nother's theorem to investigate conservation laws of non-conservative systems \cite{FredericoLazo}. The Noether theorem is usually claimed as the most important theorem for physics in the last century. All conservation laws in Physics, for examples, conservation of energy and momentum, are associated with the invariance of the action functional under continuous transformations. In contrast, dissipative forces remove energy and momentum from the system and, consequently, the standard Noether's invariants for conservative systems (as the total energy and momentum of the system) are broken. This fact and the absence of physically meaningful Lagrangian to study nonconservative systems in the context of the classical action principle proves the importance to generalise the Noether theorem to fractional calculus of variation.
Fortunately, several generalizations of the Noether theorem for Lagrangian functions depending on fractional derivatives, that can be used to study linear dissipative systems, has been recently formulated \cite{FredericoLazo,El-Nabulsi,MR2351636,MR2405377,CD:FredericoTorres:2007,Caputo,CD:FredericoTorres:2006,Zhang}.

Despite the Riewe approach has been successfully applied to study open and/or non-conservative linear systems, it cannot be directly applied to general nonlinear open systems. The limitation follows from the fact that, in order to obtain a final equation of motion containing only integer order derivatives, the Lagrangian should contain only quadratic terms depending on fractional derivatives. In the present work we proposed a generalization of a previous work \cite{LazoJerk}, and the Riewe fractional action principle, by taking advantage of a so-called practical limitation of fractional derivatives, namely, the absence of a simple chain and Leibniz's rules. We consider Lagrangian functions depending on mixed integer and fractional Caputo derivatives of differentiable functions. Although we consider only mixed integer and Caputo derivatives, because we are interested in the context of the action principle proposed in \cite{LazoCesar}, our results can be extended to several kinds of fractional derivatives. As an example, we could consider the recent $\psi$-Hilfer derivative that generalizes a wide class of fractional derivatives \cite{capelas1,capelas2}.
 Furthermore, and our main result, we formulated a Noether-type theorem for these problems, that enable us to obtain conservative quantities for nonlinear dissipative systems. As examples, we applied our generalized fractional Noether-type theorem to some nonlinear chaotic third-order dynamical systems. These systems are called jerk systems because the derivative of the acceleration with respect to time is referred to as the jerk \cite{jerk}. These systems are important because they are the simplest ever one-dimensional autonomous ordinary differential equations which display dynamical behaviors including chaotic solutions \cite{jerk1,jerk2a,jerk2b,jerk2c,jerk2d,jerk2e,jerk2f,jerk2g}. It is important to mention that jerk systems describe several phenomena in physics, engineering, and biology, such as electrical circuits, mechanical oscillators, laser physics, biological systems, etc \cite{jerk1,jerk2a,jerk2b,jerk2c,jerk2d,jerk2e,jerk2f,jerk2g}.

The article is structured in five sections. In Section \ref{sec:fdRL} we present a short introduction to the Riemann-Liouville and Caputo Fractional Calculus. Our main results are displayed in Section \ref{sec:NT}, where we generalized the Euler-Lagrange equation and the Noether theorem for a Lagrangian function that depends on fractional derivatives of differentiable functions. Illustrative examples are presented in Section \ref{sec:EX}. Finally, Section \ref{sec:C} is dedicated to our conclusions.


\section{The Riemann-Liouville and Caputo Fractional Calculus}
\lb{sec:fdRL}

We begin by recalling that there are several definitions of fractional order derivatives, which include the Riemann-Liouville, Caputo, Riesz, Weyl, etc (see \cite{OldhamSpanier,SATM,Kilbas,Diethelm,Hilfer,Magin,SKM,capelas3} for details). However, since this paper mainly addresses the Riemann-Liouville and Caputo fractional calculus, we dedicate this section to review some definitions related to them.

Actually, several known formulations of fractional calculus are somehow connected with the analytic continuation of Cauchy formula for the $n$-fold integration
\beq
\lb{a2}
\begin{split}
\int_{a}^t x(\tilde{t})(d\tilde{t})^{n} &= \int_{a}^t\int_{a}^{t_{n}}\int_{a}^{t_{n-1}}\cdots \int_{a}^{t_3}\int_{a}^{t_2} x(t_1)dt_1dt_2\cdots dt_{n-1}dt_{n} \\
&= \frac{1}{\Gamma(n)}\int_{a}^t \frac{x(u)}{(t-u)^{1-n}}du \;\;\;\;\; (n\in \mathbb{N}),
\end{split}
\eeq
where $\Gamma$ is the Euler gamma function. The proof of Cauchy formula can be found in several textbooks (for example, it can be found in \cite{OldhamSpanier}). The analytical continuation of \rf{a2} gives us a definition of integrations of non-integer order historically called Riemann-Liouville left and right fractional integrals:
\begin{definition} Let $\alpha \in \mathbb{R}_+$. The operators ${_a J^{\alpha}_t}$ and ${_t J^{\alpha}_b}$ defined on $L_1[a,b]$ by
\beq
\lb{a3}
{_a J^{\alpha}_t} x(t) =\frac{1}{\Gamma(\alpha)}\int_{a}^t \frac{x(u)}{(t-u)^{1-\alpha}}du \;\;\;\;\; (\alpha \in \mathbb{R}_+)
\eeq
and
\beq
\lb{a4}
{_t J^{\alpha}_b} x(t) =\frac{1}{\Gamma(\alpha)}\int_t^b \frac{x(u)}{(u-t)^{1-\alpha}}du \;\;\;\;\; ( \alpha \in \mathbb{R}_+),
\eeq
with $a<b$ and $a,b\in \mathbb{R}$, are called left and the right fractional Riemann-Liouville integrals of order $\alpha$, respectively.
\end{definition}

\begin{remark} 
\begin{itemize}
\item[(i)] Just recall that for any interval $[a,b]\subset\mathbb{R}$ we define
$$L_1[a,b]:=\left\{x:[a,b]\rightarrow\mathbb{R}: x \textrm{ is measurable and } \int_a^b{|x(t)|dt}<\infty\right\}.$$
\item[(ii)] For integer values of $\alpha$ the fractional Riemann-Liouville integrals \rf{a3} and \rf{a4} coincide with the usual integer order $n$-fold integration \rf{a2}. Moreover, from the definitions \rf{a3} and \rf{a4} it is easy to see that the Riemann-Liouville fractional integrals converge for any integrable function if $\alpha>0$.
\end{itemize}
\end{remark}

The integration operators ${_a J^{\alpha}_t}$ and ${_t J^{\alpha}_b}$ play a fundamental role in the definition of fractional Riemann-Liouville and Caputo calculus. In order to define the Riemann-Liouville derivatives, we recall that for positive integers $n>m$ it follows that
\beq
D^m_t x(t)=D^{n}_t {_aJ^{n-m}_t x(t)},
\eeq
where $D^m_t$ is an ordinary derivative of integer order $m$. If $\alpha\in [0,\infty)$, below we use the symbol $[\alpha]$ to represent the smallest integer that is greater or equal to $\alpha$.

\begin{definition}[Riemann-Liouville]
\begin{itemize}\item[(i)] The left Riemann-Liouville fractional derivative of order $\alpha >0$ is defined for functions in $x\in L_1[a,b]$, such that $_aJ_t^{1-\alpha}x(t)$ is $[\alpha]$ times differentiable (in the standard sense), and is defined by
\beq
\lb{a5}
{_a D^{\alpha}_t} x(t)=\frac{1}{\Gamma({[\alpha]}-\alpha)}\frac{d^{[\alpha]}}{dt^{[\alpha]}}\int_{a}^t \frac{x(u)}{(t-u)^{1+\alpha-{[\alpha]}}}du,
\eeq

\item[(ii)] The right Riemann-Liouville fractional derivative of order $\alpha >0$ is defined for functions in $x\in L_1[a,b]$, such that $_tJ_b^{1-\alpha}x(t)$ is $[\alpha]$ times differentiable (in the standard sense), and is defined by
\beq
\lb{a6}
{_t D^{\alpha}_b} x(t)=\frac{(-1)^{[\alpha]}}{\Gamma({[\alpha]}-\alpha)}\frac{d^{[\alpha]}}{dt^{[\alpha]}}\int_{t}^b \frac{x(u)}{(u-t)^{1+\alpha-{[\alpha]}}}du,
\eeq

\end{itemize}
\end{definition}

On the other hand, the Caputo fractional derivatives are defined in a distinct way.

\begin{definition}[Caputo]
\begin{itemize}\item[(i)] The left Caputo fractional derivatives of order $\alpha >0$ is defined for functions in $x\in C[a,b]$, such that $_aJ_t^{1-\alpha}x(t)$ is $[\alpha]$ times differentiable (in the standard sense), and is defined respectively by
\begin{equation}
\label{a7}
{_a^C D}^{\alpha}_t x(t) := {_a D^{\alpha}_t} [x(t)-x(a)],
\end{equation}

\item[(ii)] The right Caputo fractional derivatives of order $\alpha >0$ is defined for functions in $x\in C[a,b]$, such that $_tJ_b^{1-\alpha}x(t)$ is $[\alpha]$ times differentiable (in the standard sense), and is defined respectively by
\begin{equation}
\label{a8}
{_t^C D}^{\alpha}_b x(t)
:={_t D^{\alpha}_b} [x(b)-x(t)],
\end{equation}
\end{itemize}

Above, symbols ${_a D^{\alpha}_t}$ and ${_t D^{\alpha}_b} $ denote respectively the left and right Riemann-Liouville derivative of order $\alpha$.
\end{definition}

\begin{remark}
\begin{itemize}
\item[(i)] If $x\in C^{1}[a,b]$ and $\alpha\in(0,1)$ then we should observe that
\begin{equation*}
{_a^C D}^{\alpha}_t x(t) := {_a J^{1-\alpha}_t} x^\prime(t)\qquad\textrm{and}\qquad
{_t^C D}^{\alpha}_b x(t) := {_t J^{1-\alpha}_b} x^\prime(t).
\end{equation*}
\item[(ii)] An important consequence of definitions \eqref{a5}--\eqref{a8} is that the Riemann-Liouville and Caputo fractional derivatives are non-local operators. The left (right) differ-integration operator \eqref{a5} and \eqref{a7} (\eqref{a6} and \eqref{a8}) depends on the values of the function at left (right) of $t$, i.e. $a\leq u \leq t$ ($t\leq u \leq b$). On the other hand, it is important to note that when $\alpha$ is an integer, the Riemann-Liouville fractional derivatives \eqref{a5} and \eqref{a6} reduce to ordinary derivatives of order $\alpha$. On the other hand, in that case, the Caputo derivatives \eqref{a7} and \eqref{a8} differ from integer order ones by a polynomial of order $\alpha -1$ (see \cite{Kilbas,Diethelm} for details).
\end{itemize}
\end{remark}

It is important to observe, for the purpose of this work, that the fractional derivatives \rf{a5}--\rf{a8} do not satisfy a simple generalization of the chain and Leibniz's rules of classical derivatives \cite{OldhamSpanier,SATM,Kilbas,Diethelm,Hilfer,Magin,SKM,capelas1,capelas2}. In other words, generally we have:
\begin{equation}
\label{a9}
{_a^C D}^{\alpha}_t \left[x(t)y(t)\right]\neq y(t){_a^C D}^{\alpha}_tx(t)+x(t){_a^C D}^{\alpha}_ty(t)
\end{equation}
and
\begin{equation}
\label{a10}
{_a^C D}^{\alpha}_t y(x(t)) \neq {_a^C D}^{\alpha}_u y(u)|_{u=x} \;{_a^C D}^{\alpha}_t x(t).
\end{equation}
The absence of a simple chain and Leibniz's rules is commonly considered a practical limitation of the fractional derivatives \rf{a5}--\rf{a8}. However, in the present work, we take advantage of this limitation in order to formulate a generalized Lagrangians for nonlinear systems.

In addition to the definitions \rf{a5}--\rf{a8}, we make use of the following property in order to obtain a fractional generalization of the Euler-Lagrange condition.
\begin{theorem}[Integration by parts --- see, e.g., \cite{SKM}]
\label{thm:ml:03}
Let $0<\alpha<1$ and $x$ be a differentiable function in $[a,b]$ with $x(a)=x(b)=0$. For any function $y \in L_1([a,b])$ one has
\begin{equation}
\label{a15}
\int_{a}^{b} y(t) {_a^C D_t^{\alpha}} x(t)dt
= \int_a^b x(t) {_t D_b^{\alpha}} y(t)dt
\end{equation}
and
\begin{equation}
\label{a16}
\int_{a}^{b} y(t)  {_t^C D_b^\alpha} x(t)dt
=\int_a^b x(t) {_a D_t^\alpha} y(t) dt.
\end{equation}
\end{theorem}

It is important to notice that the formulas of integration by parts \eqref{a15} and \eqref{a16} relate Caputo left (right) derivatives to Riemann-Liouville right (left) derivatives.


\section{A Generalized Fractional Euler--Lagrange Equation and Noether-type theorem}
\label{sec:NT}

In the classical calculus of variations, it is of no conceptual and practical importance to deal with Lagrangian functions depending on derivatives of nonlinear functions of the unknown function $q$. This is due to the fact that in these cases we can always rewrite the Lagrangian $L$ as a usual Lagrangian $\tilde{L}$ by applying the chain rule. As for example, for a differentiable function $f$ we can rewrote $L(t,q,\frac{d}{dt}f(q))=L(t,q,\frac{d}{du}f(u)|_{u=q}\dot{q})=\tilde{L}(t,q,\dot{q})$, where $\frac{dq}{dt}=\dot{q}$. However, this simplification for the fractional calculus of variation is not possible due to the absence of a simple chain's rule for fractional derivatives. Actually, there are chain-type and Leibniz-type rules for some fractional derivatives \cite{OldhamSpanier,SKM,capelas2}, but they are cumbersome relations involving higher orders derivatives that cannot be simplified. It is just this apparent limitation of fractional derivatives what opens the very interesting possibility to investigate new kinds of Lagrangians suitable to study nonlinear systems. In the present work, we investigate these kinds of Lagrangians and we apply them to construct Lagrangians for some jerk systems.

\subsection{Generalized Fractional Euler--Lagrange equations}

In this subsection, we present a higher order generalization of the problem introduced in \cite{LazoJerk} by one of the present authors.

\begin{theorem}
Given two arbitrary functions
 $f,g:\mathbb{R}\rightarrow \mathbb{R}$   and let a function $q\in C^n[a,b]$ such that, $f\circ q$ and $g\circ q^{(n)}$ are differentiable in $[a,b]$. Let $S$ an action of the form
\begin{equation}
\label{t1}
S=\int_a^b {L} \left(t,q,q^{(n)},{_a^C D_t^{\alpha}} f(q),{_a^C D_t^{\alpha}} g(q^{(n)})\right)dt,
\end{equation}
where ${_a^C D_t^{\alpha}}$ is a Caputo fractional derivative of order $0<\alpha<1$, $q^{(n)}=\frac{d^n q}{dt^n}$ $(n\in\mathbb{N})$, and the function $q$ satisfy the fixed boundary conditions $q(a)=q_a$, $q(b)=q_b$, and $q^{(i)}(a)=q^{(i)}_a$ $q^{(i)}(b)=q^{(i)}_b$ $(i=1,2,...,n)$. Also let ${L}\in C^{2}[a,b]\times \mathbb{R}^{4}$. Then, the necessary condition for $S$ to possess an extremum at $q$ is that the function $q$ fulfills the following fractional Euler-Lagrange equation:
\begin{equation}
\label{t2}
\frac{\partial {L}}{\partial q}+(-1)^n\frac{d^n}{dt^n}\frac{\partial {L}}{\partial q^{(n)}}+\frac{d f}{d q}{_t D^{\alpha}_b}\frac{\partial {L}}{\partial\left({_a^C D^{\alpha}_t} f\right)}+(-1)^n\frac{d^n}{dt^n}\left(\frac{d g}{d q^{(n)}}{_t D^{\alpha}_b}\frac{\partial {L}}{\partial\left({_a^C D^{\alpha}_t} g\right)}\right)=0.
\end{equation}
\end{theorem}
\begin{proof}
In order to develop the necessary conditions for the extremum of the action \rf{t1}, we define a family of functions $q$ (weak variations)
\beq
\lb{p1}
q=q^*+\varepsilon \eta,
\eeq
where $q^*$ is the desired real function that satisfy the extremum of \rf{t1}, $\varepsilon \in \mathbb{R}$ is a constant, and the function $\eta$ defined in $[a,b]$ satisfy the boundary conditions
\beq
\lb{p2}
\eta(a)=\eta(b)=0, \;\;\;\;\;\; \eta^{(i)}(a)=\eta^{(i)}(b)=0 \;\;\;\;\;\; (i=1,2,...,n).
\eeq
The condition for the extremum is obtained when the first G\^ateaux variation is zero
\beq
\lb{p3}
\begin{split}
\delta S&=\lim_{\varepsilon \rightarrow 0} \frac{S[q^*+\varepsilon \eta]-S[q^*]}{\varepsilon}=\int_a^b \left[ \eta \frac{\partial L}{\partial q^*} +\eta^{(n)}\frac{\partial {L}}{\partial\left(q^{*(n)}\right)} \right.\\
&\left. \;\;\;\;\;\;\;\; +\left({_a^C D_t^{\alpha}}\eta\frac{df}{dq^*}\right)\frac{\partial L}{\partial \left( {_a^C D_t^{\alpha}} f\right)}+\left({_a^C D_t^{\alpha}}\eta^{(n)}\frac{dg}{dq^{*(n)}}\right)\frac{\partial L}{\partial \left( {_a^C D_t^{\alpha}} g\right)}\right]dt=0.
\end{split}
\eeq
Since the function $\eta$ satisfies both $\eta(a)=\eta(b)=0$ and $\eta^{(i)}(a)=\eta^{(i)}(b)=0$ ($i=1,2,...,n$) boundary conditions \rf{p2}, we can use the fractional integration by parts \rf{a15} and \rf{a16} in \rf{p3}, obtaining:
\beq
\lb{p4}
\begin{split}
\delta S&=\int_a^b \left[ \eta \frac{\partial L}{\partial q^*} +\eta^{(n)}\frac{\partial {L}}{\partial\left(q^{*(n)}\right)} +\eta\frac{df}{dq^*}{_t D_b^{\alpha}}\frac{\partial L}{\partial \left( {_a^C D_t^{\alpha}} f\right)}+\eta^{(n)}\frac{dg}{dq^{*(n)}}{_t D_b^{\alpha}}\frac{\partial L}{\partial \left( {_a^C D_t^{\alpha}} g\right)}\right]dt\\
&=\int_a^b \eta \left[\frac{\partial L}{\partial q^*} +(-1)^n\frac{d^n}{dt^n}\frac{\partial {L}}{\partial\left(q^{*(n)}\right)} +\frac{df}{dq^*}{_t D_b^{\alpha}}\frac{\partial L}{\partial \left( {_a^C D_t^{\alpha}} f\right)}\right.\\
&\qquad \qquad \qquad \qquad \qquad \qquad \qquad \left.+(-1)^n\frac{d^n}{dt^n}\left(\frac{dg}{dq^{*(n)}}{_t D_b^{\alpha}}\frac{\partial L}{\partial \left( {_a^C D_t^{\alpha}} g\right)}\right)\right]dt=0,
\end{split}
\eeq
where additional usual integration by parts was performed in the terms containing $\eta^{(n)}$. Finally, by using the Fundamental Lemma of the calculus of variations, we obtain the fractional Euler-Lagrange equations \rf{t2}.
\end{proof}

It is important to notice that our Theorem can be easily extended for Lagrangians depending on left Caputo derivatives, and Riemann-Liouville fractional derivatives. Actually, it is also easy to generalize in order to include a nonlinear function $g\left({_a^C D_t^{\alpha}} x\right)$ instead of $g(\dot{x})$. Finally, it is important to mention that our Theorem generalizes \cite{Riewe1,Riewe2} and the more general formulation proposed in \cite{Agrawal}, as well as the Lagrangian formulation for higher order linear open systems \cite{LazoCesar}
(for a review in recent advances in the calculus of variations with fractional derivatives see \cite{MalinowskaTorres}).

\begin{remark} For $n=1$, and $f(q)=g(\dot{q})=0$, our condition \rf{t2} reduces to the ordinary Euler-Lagrange equation, and the boundary conditions $q(a)=q_a$, $q(b)=q_b$ and $\dot{q}(a)=\dot{q}_a$, $\dot{q}(b)=\dot{q}_b$ are defined by only two arbitrary parameter. Note that for this particular case the Euler-Lagrange equation is a second order ordinary differential equation whose solution $q(t)$ is fixed by two parameters. For example, by imposing the conditions $q(a)=q_a$ and $q(b)=q_b$ the solution $q(t)$ is fixed and, consequently, the numbers $\dot{q}_a$ and $\dot{q}_b$ are automatically fixed as functions of $q_a$ and $q_b$.
\end{remark}

\subsection{Fractional Noether-type theorem}

In a seminal work, Emmy Noether obtained a correspondence linking symmetry groups of a Lagrangian system and its corresponding constants of motion (also known as its first integrals). Actually, when a Lagrangian is invariant under the action of a symmetry group, then the system exhibits a corresponding conservation law. A symmetry is defined by a one-parameter group of diffeomorphisms under which the action functional is invariant:

\begin{definition}(Symmetry group)
\label{sy} Let $\varepsilon \in \mathbb{R}$ be a parameter. If $\phi(\varepsilon,\cdot)\,:\mathbb{R}^n\rightarrow\mathbb{R}^n$ is a diffeomorphism, assuming that $\phi$ is $C^2$ with respect to $\varepsilon$, and if we have
\begin{enumerate}
\item $\phi (0,\cdot) = I_{\mathbb{R}^n}$;
\item $\forall \varepsilon,\varepsilon' \in \mathbb{R},
\; \phi (\varepsilon,\cdot) \circ \phi (\varepsilon',\cdot) = \phi
(\varepsilon+\varepsilon',\cdot) $,
\end{enumerate}
then $\Phi=\left\{\phi(\varepsilon,\cdot)\right\}_{\varepsilon\in\mathbb{R}}$
is a one-parameter group of diffeomorphisms of $\mathbb{R}^n$.
\end{definition}

\begin{remark} The translation in a spatial direction $u$ is a typical case of a one-parameter group of diffeomorphisms: 
$$\phi:q \longmapsto q+\varepsilon u\,,\quad q,u\in\mathbb{R}^n.$$
Another common example is the rotations by an angle $\theta$
$$\phi:q \longmapsto qe^{i\varepsilon\theta}\,,\quad q\in\mathbb{C}\,,\quad \theta\in \mathbb{R}\,.$$
\end{remark}

In the present work, we use the concept of a group of diffeomorphisms rather than the concept of infinitesimal transformations as in \cite{MR2351636,MR2405377}. These two concepts can be related by a Taylor expansion of $\phi(\varepsilon,q(t))$ in the neighborhood of $\varepsilon=0$. We have:
$$\phi(\varepsilon,q(t))=\phi(0,q(t))+\varepsilon\frac{\partial \phi}{\partial \varepsilon}(0,q(t))+o(\varepsilon)\,.$$
By considering the premise in the Definition \ref{sy} that $\phi (0,\cdot) = I_{\mathbb{R}^n}$, we obtain the related one-parameter infinitesimal transformation
$$q(t)\longmapsto q(t) + \varepsilon\xi(t,q(t)) + o(\varepsilon),$$
where we write $\xi(t,q(t))=\frac{\partial \phi}{\partial
\varepsilon}(0,q(t))$.

With the aim of generalizing the Noether Theorem for our action functional
\eqref{t1}, we follows closely the approach described in \cite{CD:Aves:1986} and used in our previous work \cite{FredericoLazo}. Firstly, we need to define the concept  of  symmetry for a functional of the form \eqref{t1} under the action of a group of diffeomorphisms. Let us consider first the simpler case where we have transformations only in space (without time transformation)
\begin{definition}(Invariance without time transformation).
\label{def:inv1:MR} Let $\Phi_1=\left\{\phi_1(\varepsilon,\cdot)\right\}_{\varepsilon\in\mathbb{R}}$ be a one-parameter group of diffeomorphisms of $\mathbb{R}$ with $\overline{q}=\phi_1(\varepsilon,q)$. Then, the fractional functional \eqref{t1} is $\varepsilon$-invariant under the action of $\Phi_1=\left\{\phi_1(\varepsilon,\cdot)\right\}_{\varepsilon\in\mathbb{R}}$ if, for any solution $q$ of \eqref{t2},  it satisfies
\begin{multline}
\label{eq:invdf}
\int_{t_a}^{t_b} L\left(t,q,\dot{q},{_a^CD_t^\alpha f(q)},{_a^CD_t^\alpha g(\dot{q})}\right) dt=\int_{t_a}^{t_b} L\left(t,\overline{q},\dot{\overline{q}},{_a^CD_t^\alpha f(\overline{q})},{_a^CD_t^\alpha g(\dot{\overline{q}})}\right) dt\\
= \int_{t_a}^{t_b}
L\left(t,\phi_1(\varepsilon,q),\dot{\phi_1}(\varepsilon,q),{_a^CD_t^\alpha
f(\phi_1(\varepsilon,q))},{_a^CD_t^\alpha
g(\dot{\phi_1}(\varepsilon,q))}\right) dt
\end{multline}
in any subinterval $[{t_{a}},{t_{b}}] \subseteq [a,b]$. Here and in the sequel, we consider $n=1\,$.
\end{definition}

From the definition of invariance in \eqref{eq:invdf}, we obtain an important necessary condition that will play a crucial role in obtaining a generalized Noether-type theorem. This condition is given by the following theorem:
\begin{theorem}(Necessary condition for invariance).
Let the functional \eqref{t1} be invariant in the meaning of
\textup{Definition~\ref{def:inv1:MR}}. Then
\begin{multline}
\label{eq:cnsidf11} \frac{\partial \phi_1}{\partial
\varepsilon}(0,q(t))\cdot
\frac{d}{dt}\partial_{3} L\left(t,q,\dot{q},{_a^CD_t^\alpha f(q)},{_a^CD_t^\alpha g(\dot{q})}\right)\\
+ \partial_{3}L\left(t,q,\dot{q},{_a^CD_t^\alpha f(q)},{_a^CD_t^\alpha g(\dot{q})}\right)\cdot
\frac{d}{dt}\frac{\partial \phi_1}{\partial \varepsilon}(0,q(t))\\
+ \partial_{4} L\left(t,q,\dot{q},{_a^CD_t^\alpha f(q)},{_a^CD_t^\alpha g(\dot{q})}\right)
\cdot {_a^CD_t^\alpha}\left(\frac{df}{dq}\frac{\partial\phi_1}{\partial\varepsilon}(0,q)\right)\\
- \frac{df}{dq}\frac{\partial \phi_1}{\partial \varepsilon}(0,q(t))\cdot
{_tD_b^\alpha}\partial_{4} L\left(t,q,\dot{q},{_a^CD_t^\alpha f(q)},{_a^CD_t^\alpha g(\dot{q})}\right)\\
+ \frac{\partial \phi_1}{\partial
\varepsilon}(0,q(t))\cdot \frac{d}{dt}\left(\frac{d g}{d\dot{q}}{_t D^{\alpha}_b}\partial_{5} L\left(t,q,\dot{q},{_a^CD_t^\alpha f(q)},{_a^CD_t^\alpha g(\dot{q})}\right)\right)\\
+\frac{d}{dt}\frac{\partial\phi_1}{\partial\varepsilon}(0,q)\frac{dg}{d\dot{q}}\cdot{_t^CD_b^\alpha}\partial_{5} L\left(t,q,\dot{q},{_a^CD_t^\alpha f(q)},{_a^CD_t^\alpha g(\dot{q})}\right)
 = 0 \, .
\end{multline}
\end{theorem}

\begin{proof}
In order to proof the condition given by \eqref{eq:cnsidf11}, since the functional is invariant, we first differentiate \eqref{eq:invdf} with respect to the variable $\varepsilon$. After this, by setting $\varepsilon=0$ we obtain:
\begin{multline}
\label{eq:SP} 0 = \int_{t_a}^{t_b}\Biggl[\partial_{2} L\left(t,q,\dot{q},{_a^CD_t^\alpha f(q)},{_a^CD_t^\alpha g(\dot{q})}\right)\cdot\frac{\partial \phi_1}{\partial \varepsilon}(0,q)\\
+\partial_{3}L\left(t,q,\dot{q},{_a^CD_t^\alpha f(q)},{_a^CD_t^\alpha g(\dot{q})}\right)\cdot\frac{\partial}{\partial\varepsilon}\left.\left[ \frac{d \phi_1}{dt}(\varepsilon,q)\right]\right|_{\varepsilon=0} \\
+ \partial_{4} L\left(t,q,\dot{q},{_a^CD_t^\alpha f(q)},{_a^CD_t^\alpha g(\dot{q})}\right)
\cdot\frac{\partial}{\partial\varepsilon}\left.\left[ {_a^CD_t^\alpha}
f(\phi_1(\varepsilon,q))\right]\right|_{\varepsilon=0}\\
+ \partial_{5} L\left(t,q,\dot{q},{_a^CD_t^\alpha f(q)},{_a^CD_t^\alpha g(\dot{q})}\right)
\cdot\frac{\partial}{\partial\varepsilon}\left.\left[ {_a^CD_t^\alpha}g
\left((\dot{\phi}_1(\varepsilon,q))\right)\right]\right|_{\varepsilon=0}\Biggr]dt\,.
\end{multline}
The second term in the integral \eqref{eq:SP} can be simplified by recalling that we consider $\phi_1(\varepsilon,q)\in C^2$ with respect to $\varepsilon$ in the Definition~\ref{sy}. Consequently, we have
 \begin{equation}
 \label{ce}
 \frac{\partial}{\partial\varepsilon}\left.\left[ \frac{d \phi_1}{dt}(\varepsilon,q)\right]\right|_{\varepsilon=0}
 =\frac{d}{dt}\frac{\partial\phi_1}{\partial\varepsilon}(0,q)\,.
\end{equation}
The third and fourth terms in this integral can also be simplified in the same way as in \eqref{ce}, since $_a^CD_t^\alpha$ operates only on $t$. Then, we can reduce
 \begin{equation}
 \frac{\partial}{\partial\varepsilon}\left.\left[ {_a^CD_t^\alpha} f(\phi_1(\varepsilon,q))\right]\right|_{\varepsilon=0}
 ={_a^CD_t^\alpha}\left(\frac{df}{dq}\frac{\partial\phi_1}{\partial\varepsilon}(0,q)\right)\,,
 \end{equation}
 and
 \begin{equation}
 \label{ce1}
 \frac{\partial}{\partial\varepsilon}\left.\left[ {_a^CD_t^\alpha} g(\dot{\phi}_1(\varepsilon,q))\right]\right|_{\varepsilon=0}
 ={_a^CD_t^\alpha}\left(\frac{dg}{d\dot{q}}\frac{d}{dt}\frac{\partial\phi_1}{\partial\varepsilon}(0,q)\right)\,.
 \end{equation}

By inserting \eqref{ce}--\eqref{ce1} into relation \eqref{eq:SP}, and by using the Euler--Lagrange equation \eqref{t2}, we obtain
\begin{multline*}
\int_{t_a}^{t_b}\Biggl[\frac{\partial \phi_1}{\partial
\varepsilon}(0,q(t))\cdot
\frac{d}{dt}\partial_{3} L\left(t,q,\dot{q},{_a^CD_t^\alpha f(q)},{_a^CD_t^\alpha g(\dot{q})}\right)\\
+ \partial_{3}L\left(t,q,\dot{q},{_a^CD_t^\alpha f(q)},{_a^CD_t^\alpha g(\dot{q})}\right)\cdot
\frac{d}{dt}\frac{\partial \phi_1}{\partial \varepsilon}(0,q(t))\\
+ \partial_{4} L\left(t,q,\dot{q},{_a^CD_t^\alpha f(q)},{_a^CD_t^\alpha g(\dot{q})}\right)
\cdot {_a^CD_t^\alpha}\left(\frac{df}{dq}\frac{\partial\phi_1}{\partial\varepsilon}(0,q)\right)\\
- \frac{df}{dq}\frac{\partial \phi_1}{\partial \varepsilon}(0,q(t))\cdot
{_tD_b^\alpha}\partial_{4} L\left(t,q,\dot{q},{_a^CD_t^\alpha f(q)},{_a^CD_t^\alpha g(\dot{q})}\right)\\
+ \frac{\partial \phi_1}{\partial
\varepsilon}(0,q(t))\cdot \frac{d}{dt}\left(\frac{d g}{d\dot{q}}{_t D^{\alpha}_b}\partial_{5} L\left(t,q,\dot{q},{_a^CD_t^\alpha f(q)},{_a^CD_t^\alpha g(\dot{q})}\right)\right)\\
+\partial_{5} L\left(t,q,\dot{q},{_a^CD_t^\alpha f(q)},{_a^CD_t^\alpha g(\dot{q})}\right)\cdot{_a^CD_t^\alpha}\left(\frac{dg}{d\dot{q}}\frac{d}{dt}\frac{\partial\phi_1}{\partial\varepsilon}(0,q)\right)\Biggr]dt= 0 \, .
\end{multline*}
Finally, by using Theorem~\ref{thm:ml:03} in the last term in above integral, and by having in mind that condition \eqref{eq:invdf} is valid for any subinterval
$[{t_{a}},{t_{b}}] \subseteq [a,b]$, we obtain equation \eqref{eq:cnsidf11}.
\end{proof}

Before we enunciate our generalized Noether-type theorem, we present here a Lemma introduced in \cite{CD:BoCrGr} known as Transfer Formula:
\begin{lemma}(Transfer Formula (see \cite{CD:BoCrGr})).
\label{trfo} \\Let $f,g\in
C^{\infty}\left([a,b];\mathbb{R}^n\right)$ be functions that satisfy the condiction that the sequences $\left(g^{(k)}\cdot
_aI_t^{k-\alpha}(f-f(a))\right)_{k\in \mathbb{N}\setminus\{0\}}$ and
$\left(f^{(k)}\cdot _tI_b^{k-\alpha}g\right)_{k\in
\mathbb{N}\setminus\{0\}}$ converge to $0$ uniformly in the interval $[a,b]$ (in the sequel, we are going to refer to this condition as condition $(\mathcal{C})$).
Thus, the equality
\begin{equation*}
g\cdot {_a^CD_t^\alpha}f-f\cdot{_tD_b^\alpha}g
=\frac{d}{dt}\left[\sum_{r=0}^{\infty}\left((-1)^{r}g^{(r)} \cdot
{_aI_t}^{r+1-\alpha}(f-f(a))+f^{(r)} \cdot
{_tI_b}^{r+1-\alpha}g\right)\right]
\end{equation*}
holds.
\end{lemma}
This Lemma will be used to proof our first generalized Noether-type theorem.

\begin{theorem}(Fractional Noether-type theorem without time transformation).
\label{theo:tnadf1} Let $S$ be a functional given by \eqref{t1} with $n=1$. If $S$ is invariant in the meaning of \textup{Definition~\ref{def:inv1:MR}}, and if the functions $\frac{\partial
\phi_1}{\partial \varepsilon}(0,q)$ and $\partial_{4} L$ satisfy the
condition $(\mathcal{C})$ in \textup{Lemma~\ref{trfo}}, then the equality
\begin{multline}\label{TN}
\frac{d}{dt}\Biggl[f_2\left(\partial_{3}L+\frac{d g}{d\dot{q}}{_t D^{\alpha}_b}\partial_{5} L\right)
+\sum_{r=0}^{\infty}\Bigl((-1)^{r}\partial_{4} L^{(r)}
\cdot {_aI_t}^{r+1-\alpha}(f_2-f_2(a))\\
+f_2^{(r)}\cdot {_tI_b}^{r+1-\alpha}\partial_{4} L\Bigr)\Biggr] = 0,
\end{multline}
where $f_2=\frac{\partial \phi_1}{\partial
\varepsilon}(0,q)$, holds for any solutions $q(\cdot)$, $t \in [a,b]$ of the Euler--Lagrange equation \eqref{t2}.
\end{theorem}

\begin{proof}
The proof of our fractional Noether-type theorem without time transformation follows directly by using Lemma~\ref{trfo} into the equation \eqref{eq:cnsidf11}.
\end{proof}

In Definition \ref{def:inv1:MR} and in the Theorem \ref{theo:tnadf1} we consider only the particular case without time transformation. Let us consider now the more general case including both space and time transformations.  Firstly, we should formulate a general notion of invariance for the functional \eqref{t1} under the action of a one-parameter group of diffeomorphisms including both space and time transformations:

\begin{definition}(General invariance).
\label{def:invadf}
Let $\Phi_{i=1,2}=\left\{\phi_i(\varepsilon,\cdot)\right\}_{\varepsilon\in\mathbb{R}}$ be a one-parameter group of diffeomorphisms of $\mathbb{R}^{2}$ with $\overline{q}=\phi_1(\varepsilon,q)$ and $\bar{t}=\phi_2(\varepsilon,t)$. The functional \eqref{t1} is $\varepsilon$-invariant under the action of $\Phi_{i=1,2}=\left\{\phi_i(\varepsilon,\cdot)\right\}_{\varepsilon\in\mathbb{R}}$ if it satisfies
\begin{multline}
\label{eq:invdf1}
\int_{t_{a}}^{t_{b}}L \left(t,q,\dot{q},{_a^CD_t^\alpha f(q)},{_a^CD_t^\alpha g(\dot{q})}\right) dt=\int_{\bar{t}_{\bar{a}}}^{\bar{t}_{\bar{b}}}L \left(\bar{t},\bar{q},\dot{\bar{q}},{_{\bar{a}}^CD_{\bar{t}}^\alpha f(\bar{q})},{_{\bar{a}}^CD_{\bar{t}}^\alpha g(\dot{\bar{q}})}\right) d\bar{t}\\
= \int_{\phi_2(\varepsilon,t_{a})}^{\phi_2(\varepsilon,t_{b})}
L\left(\phi_2(\varepsilon,t),\phi_1(\varepsilon,q(t)),\frac{\dot{\phi_1}(\varepsilon,q(t))}{\dot{\phi_2}(\varepsilon,t)},{_{\bar{a}}^CD_{\bar{t}}^\alpha
f(\phi_1(\varepsilon,q(t)))},{_{\bar{a}}^CD_{\bar{t}}^\alpha
g(\dot{\phi_1}(\varepsilon,q(t)))}\right)\dot{\phi_2}(\varepsilon,q(t))dt
\end{multline}
for any solution $q(\cdot)$ of \eqref{t2}, and for any subinterval $[{t_{a}},{t_{b}}] \subseteq [a,b]$. In \eqref{eq:invdf1} we denote $\dot{\bar{q}}=\frac{d\bar{q}}{d\bar{t}}$, $\dot{\phi_i}=\frac{d\phi_i}{dt}$  ($i=1,2$), and $\bar{a}=\phi_2(\varepsilon,t_{a})$.
\end{definition}

We can now extend our previous Theorem \ref{theo:tnadf1} for the more general case including both space and time transformations. The generalized Noether-type theorem for fractional variational problems depending on fractional derivatives of functions is given by:
\begin{theorem}(Noether-type theorem with fractional derivatives of functions).
\label{theo:tndf} 
Let $S$ be a functional given by \eqref{t1} with $n=1$. If $S$ is invariant in the meaning of \textup{Definition~\ref{def:invadf}}, and if $\frac{\partial
\phi_1}{\partial \varepsilon}(0,q)$ and $\partial_{4} L$ satisfy the
condition $(\mathcal{C})$ in \textup{Lemma~\ref{trfo}}, then the equality
\begin{multline}
\label{eq:LC:Frac:RL1} \frac{d}{dt}\Biggl[f_2\left(\partial_{3}L+\frac{d g}{d\dot{q}}{_t D^{\alpha}_b}\partial_{5} L\right)
+\sum_{r=0}^{\infty}\Bigl((-1)^{r}\partial_{4} L^{(r)}
\cdot {_aI_t}^{r+1-\alpha}(f_2-f_2(a))
+f_2^{(r)}\cdot {_tI_b}^{r+1-\alpha}\partial_{4} L\Bigr)\\
+ \tau\Bigl(L-\dot{q}\cdot\partial_{3} L -\alpha\left(\partial_{4}
L\cdot{_a^CD_t^\alpha f(q)}+\partial_{5}
L\cdot{_a^CD_t^\alpha g(\dot{q})}\right)-\partial_{5}
L\cdot{_a^CD_t^\alpha \dot{q}\frac{dg}{d\dot{q}}}\Bigr)\Biggr]= 0
\end{multline}
holds along any solutions $q(\cdot)$, $t \in [a,b]$ of  the Euler--Lagrange equation \eqref{t2}. In \eqref{eq:LC:Frac:RL1} we denote $f_2=\frac{\partial \phi_1}{\partial
\varepsilon}(0,q)$ and $\tau=\frac{\partial \phi_2}{\partial
\varepsilon}(0,t)$.
\end{theorem}

\begin{proof}
In order to proof our generalized Noether-type theorem, we extend the procedure used in our previous work \cite{FredericoLazo} and in \cite{CD:Aves:1986}.
Firstly, we parametrize the time $t\in [a,b]$ by using a Lipschitz transformation as follow:
\begin{equation*}
t\longmapsto \sigma h(\lambda) \in [\sigma_{a},\sigma_{b}],\,\, \mbox{with} \,\, t(\sigma_{a})=a,\,\, t(\sigma_{b})=b, \end{equation*}
where we impose the condition 
\begin{equation}
\label{eq:condla} t_{\sigma}^{'}|_{\lambda=0} =\left.\frac{dt(\sigma)}{d\sigma}\right|_{\lambda=0}=h(0) = 1.
\end{equation}
By using this parametrization, the functional $S[q(\cdot)]$ defined in \eqref{t1} can be rewritten as an autonomous functional $\bar{S}[t(\cdot),q(t(\cdot))]$ given by
\begin{multline}
\label{eq:tempo}
S[q(\cdot)]\longmapsto \bar{S}[t(\cdot),q(t(\cdot))]\\
= \int_{\sigma_{a}}^{\sigma_{b}}
L\left(t(\sigma),q(t(\sigma)),\dot{q}(t(\sigma)),
{_{\sigma_{a}}^CD_{t(\sigma)}^{\alpha}f(q(t(\sigma)))},{_{\sigma_{a}}^CD_{t(\sigma)}^{\alpha}g(\dot{q}(t(\sigma)))}
\right)t_{\sigma}^{'} d\sigma.
\end{multline}
We now need to rewrite the Caputo fractional derivatives in \eqref{eq:tempo}. From the definitions presented in Section \ref{sec:fdRL} we obtain
\begin{equation}
\lb{eq:cap01}
\begin{split}
_{\sigma_{a}}^C&D_{t(\sigma)}^{\alpha}f(q(t(\sigma)))\hspace*{-0.5cm}\\
&= \frac{1}{\Gamma(1-\alpha)}\frac{d}{dt(\sigma)} \int_{\frac{a}{h(\lambda)}}^{\sigma
h(\lambda)}\left({\sigma
h(\lambda)}-\theta\right)^{-\alpha}
\left[f\left(q\left(\theta h^{-1}(\lambda)\right)\right)-f\left(q\left(\frac{a}{h(\lambda)}\right)\right)\right]d\theta    \\
&= \frac{(t_{\sigma}^{'})^{-\alpha}}{\Gamma(1-\alpha)}\frac{d}{d\sigma}
\int_{\frac{a}{(t_{\sigma}^{'})^{2}}}^{\sigma}
(\sigma-s)^{-\alpha}\left(
f(q(s))-f\left(q\left(\frac{a}{(t_{\sigma}^{'})^2}\right)\right)\right)ds  \\
&= (t_{\sigma}^{'})^{-\alpha}\,\,{{^C_\chi
D_{\sigma}^{\alpha}f(q(\sigma))}},\,\left(
\chi={\frac{a}{(t_{\sigma}^{'})^{2}}}\right)\, ,
\end{split}
\end{equation}
and, by a similar development,
\begin{equation}
\lb{eq:cap02}
_{\sigma_{a}}^CD_{t(\sigma)}^{\alpha}g(\dot{q}(t(\sigma)))=(t_{\sigma}^{'})^{-\alpha}\,\,{{^C_\chi
D_{\sigma}^{\alpha}g\left(\frac{q_{\sigma}^{'}}{t_{\sigma}^{'}}\right)}} \, .
\end{equation}
Inserting \eqref{eq:cap01} and \eqref{eq:cap02} into \eqref{eq:tempo} we get
\begin{equation*}
\begin{split}
&\bar{S}[t(\cdot),q(t(\cdot))] \\
&= \int_{\sigma_{a}}^{\sigma_{b}}
L\left(t(\sigma),q(t(\sigma)),\frac{q_{\sigma}^{'}}{t_{\sigma}^{'}},(t_{\sigma}^{'})^{-\alpha}\,\,{^C_\chi
D_{\sigma}^{\alpha}f(q(\sigma))},(t_{\sigma}^{'})^{-\alpha}\,\,{{^C_\chi
D_{\sigma}^{\alpha}g\left(\frac{q_{\sigma}^{'}}{t_{\sigma}^{'}}\right)}}\right) t_{\sigma}^{'} d\sigma \\
&\doteq \int_{\sigma_{a}}^{\sigma_{b}}
\bar{L}_{f}\left(t(\sigma),q(t(\sigma)),q_{\sigma}^{'},t_{\sigma}^{'},{^C_\chi
D_{\sigma}^{\alpha}f(q(\sigma))},{{^C_\chi
D_{\sigma}^{\alpha}g\left(\frac{q_{\sigma}^{'}}{t_{\sigma}^{'}}\right)}}\right)d\sigma \\
&= \int_a^b L \left(t,q,\dot{q},{_a^CD_t^\alpha f(q)},{_a^CD_t^\alpha g(\dot{q})}\right) dt \\
&= S[q(\cdot)] \, .
\end{split}
\end{equation*}
Since $S[q(\cdot)]$ defined in \eqref{t1} is an invariant functional in the meaning of
Definition~\ref{def:invadf}, the autonomous functional $\bar{S}[t(\cdot),q(t(\cdot))]$ in
\eqref{eq:tempo} will be invariant in the meaning of
Definition~\ref{def:inv1:MR}. Consequently, from
Theorem~\ref{theo:tnadf1} we have
\begin{multline}
\label{eq:tnadf2} \frac{d}{dt}\Biggl[f_2\left(\partial_{3}
\bar{L}_{f}+\frac{dg}{d(\frac{q_{\sigma}^{'}}{t_{\sigma}^{'}})}
{^C_{\sigma}
D_{\nu}^{\alpha}\partial_{6}
\bar{L}_{f}}\right)\\+\tau\frac{\partial}{\partial t'_\sigma} \bar{L}_{f}
+\sum_{r=0}^{\infty}\Bigl((-1)^{r}
\partial_{5} \bar{L}_{f}^{(r)}\cdot {_aI_t}^{r+1-\alpha}(f_2-f_2(a))
+f_2^{(r)}\cdot {_tI_b}^{r+1-\alpha}\partial_{5}
\bar{L}_{f}\Bigr)\Biggr]= 0 \, ,
\end{multline}
where $\nu=\frac{b}{(t_{\sigma}^{'})^{2}}\,.$

The last step in our proof consist in set $\lambda = 0$. In this case, from \eqref{eq:condla} we get
\begin{equation*}
\begin{cases}
{^C_\chi D_{\sigma}^{\alpha}f(q(\sigma))}
 = {_a^CD_t^\alpha f(q(t))}\\
 _{\sigma_{a}}^CD_{t(\sigma)}^{\alpha}g(\dot{q}(t(\sigma)))={_a^CD_t^\alpha g(\dot{q}(t))}\\
{^C_{\sigma}
D_{\nu}^{\alpha}\partial_{6}
\bar{L}_{f}}={^C_t
D_{b}^{\alpha}\partial_{5}
L}\\
\frac{dg}{d(\frac{q_{\sigma}^{'}}{t_{\sigma}^{'}})}=\frac{dg}{d\dot{q}}
 \end{cases}
\end{equation*}
and, consequently,
\begin{equation}
\label{eq:prfMR:q1}
\begin{cases}
\partial_{3}\bar{L}_{f}
=\partial_{3} L,\\
\partial_{5}\bar{L}_{f}
=\partial_{4} L,
\end{cases}
\end{equation}
and
\begin{equation}
\label{eq:prfMR:q2}
\begin{split}
&\frac{\partial}{\partial t'_\sigma} \bar{L}_{f} = L +
\partial_{3}{\bar{L}_{f}}
\frac{\partial}{\partial t_{\sigma}^{'}}\left(
\frac{q_{\sigma}^{'}}{t_{\sigma}^{'}}\right)t_{\sigma}^{'}+\partial_{5}{\bar{L}_{f}}
 \frac{\partial}{\partial t_{\sigma}^{'}}\left[(t_{\sigma}^{'})^{-\alpha}\,\,{^C_\chi
D_{\sigma}^{\alpha}f(q(\sigma))}\right]t_{\sigma}^{'}\\
&+\partial_{6}{\bar{L}_{f}}
 \frac{\partial}{\partial t_{\sigma}^{'}}\left[(t_{\sigma}^{'})^{-\alpha}\,\,{{^C_\chi
D_{\sigma}^{\alpha}g\left(\frac{q_{\sigma}^{'}}{t_{\sigma}^{'}}\right)}}\right]t_{\sigma}^{'}\\
&=L-\dot{q}\cdot\partial_{3} L -\alpha\left(\partial_{4}
L\cdot{_a^CD_t^\alpha f(q)}+\partial_{5}
L\cdot{_a^CD_t^\alpha g(\dot{q})}\right)-\partial_{5}
L\cdot{_a^CD_t^\alpha \dot{q}\frac{dg}{d\dot{q}}}  \, .
\end{split}
\end{equation}
Finally, by inserting \eqref{eq:prfMR:q1} and \eqref{eq:prfMR:q2} into relation \eqref{eq:tnadf2} we obtain \eqref{eq:LC:Frac:RL1}.
\end{proof}

An important particular case of our generalized Noether-type theorem~\ref{theo:tndf} is obtained when we consider autonomous problems. In this case, if the Lagrangian function does not depend explicitly on the time $t$, our variational problem reduces to
\begin{gather}
S[q(\cdot)] =\int_a^b L\left(q,\dot{q},{_a^CD_t^\alpha f(q)},{_a^CD_t^\alpha g(\dot{q})}\right) dt
\longrightarrow \min. \label{eq:FOCP:CO4}
\end{gather}
Furthermore, in such circumstances, from our generalized Noether-type theorem we obtain the following important corollary
\begin{corollary}(Conservation law under time translations for autonomous problems).
\label{cor:FOCP:CV}
Let $S$ be an autonomous functional given by the problem \eqref{eq:FOCP:CO4}, and let us consider the time translation transformation given by $\phi_1(\varepsilon,q(t))= q(t)$ and $\phi_2(\varepsilon,t) = t+\varepsilon$. If $S$ is invariant in the meaning of \textup{Definition~\ref{def:invadf}}, then the equality
\begin{equation}
\label{eq:ATT}
 \frac{d}{dt}\Bigl(L-\dot{q}\cdot\partial_{3} L -\alpha\left(\partial_{4}
L\cdot{_a^CD_t^\alpha f(q)}+\partial_{5}
L\cdot{_a^CD_t^\alpha g(\dot{q})}\right)-\partial_{5}
L\cdot{_a^CD_t^\alpha \dot{q}\frac{dg}{d\dot{q}}}\Bigr)= 0
\end{equation}
holds along any solutions $q(\cdot)$ of the Euler--Lagrange equation \eqref{t2}, and for any $t \in [a,b]$.

\end{corollary}

\begin{proof}
The proof follows from our more general Theorem \ref{theo:tndf}. Firstly, it is easy to see that the autonomous functional \eqref{eq:FOCP:CO4} is invariant under time translations since, in this case, we have $\frac{d\phi_2}{dt}(\varepsilon,q(t)) = 1$ in \eqref{eq:invdf1}, and the Lagrangian function $L$ does not depend explicity on $t$. Consequently, the condition \eqref{eq:LC:Frac:RL1} holds along any solutions $q(\cdot)$ of the Euler--Lagrange equation \eqref{t2} for the autonomous problem \eqref{eq:FOCP:CO4}. In order to complete the proof, we should show that, in this case, \eqref{eq:LC:Frac:RL1} reduces to \eqref{eq:ATT}. This can be done by observing that $f_2=\frac{\partial \phi_1}{\partial
\varepsilon}(0,q)=0$, $\tau=\frac{\partial \phi_2}{\partial
\varepsilon}(0,t)=1$, and 
\begin{equation*}
\begin{split}
{_{\bar{a}}^CD_{\bar{t}}^\alpha
f(\phi_1(\varepsilon,q(t))})
&=\frac{1}{\Gamma(1-\alpha)}\frac{d}{dt }\int_{\bar{a}}^{\bar{t}}
(\bar{t}-\theta)^{-\alpha}(
f(\phi_1(\varepsilon,q(\theta)))-f(q(\bar{a})))d\theta
\\
&=\frac{1}{\Gamma(1-\alpha)}\frac{d}{dt } \int_{a+\varepsilon}^{t+\varepsilon}
(t+\varepsilon-\theta)^{-\alpha}(
f(\phi_1(\varepsilon,q(\theta)))-f(q(a+\varepsilon)))d\theta\\
&=\frac{1}{\Gamma(1-\alpha)}\frac{d}{dt } \int_{a}^{t}
(t-s)^{-\alpha}\left(
f(\phi_1(\varepsilon,q(s+\varepsilon)))-f(q(a)\right)ds\\
&={_a^CD^\alpha_{t}f(\phi_1(\varepsilon,q(t+\varepsilon)))}={_a^CD^\alpha_{t}f(\phi_1(\varepsilon,q(t})))\\
&={_a^CD^\alpha_{t}f({q}(t))}\, .
\end{split}
\end{equation*}
\end{proof}

The importance of Corollary \ref{cor:FOCP:CV} to study nonconservative and nonlinear systems will be displayed in the next section where we analyze some examples of nonlinear chaotic jerk systems. Furthermore, it is important to notice that in the integer order derivative case ($\alpha=1$) the condition \eqref{eq:ATT} is reduced to the well-known law of conservation of the energy $E$
\begin{equation*}
 \frac{d}{dt}\Bigl(L-\dot{q}\cdot\frac{\partial L}{\partial \dot{q}}\Bigr)= 0\,\, \longrightarrow \,\, E=L-\dot{q}\cdot\frac{\partial L}{\partial \dot{q}}=\mbox{constant}
\end{equation*}
of classical mechanics.


\section{Lagrangian For Nonlinear Chaotic Jerk Systems}
\label{sec:EX}

As examples for applications of our generalized fractional Noether-type theorem \ref{theo:tndf}, in this section we obtain conserved quantities for some jerk systems. In particular, we consider three cases. In the first two, we investigate the conserved quantities related to the symmetries under time translation. The third example illustrates a conservation law related to internal (global) symmetry.

\subsection{Example 1: nonlinearity depending only on $x$}

Let us consider first some autonomous systems. The simplest one-dimensional family of jerk systems that displays chaotic solutions is given by \cite{jerk,jerk1,jerk2a,jerk2b,jerk2c,jerk2d,jerk2e,jerk2f,jerk2g}
\beq
\lb{c5}
\dddot{x}+A\ddot{x}+\dot{x}=G(x),
\eeq
where $A$ is a system parameter, and $G(x)$  is a nonlinear function containing one nonlinearity, one system parameter, and a constant term. We can formulate a Lagrangian for this jerk system by
\begin{equation}
\lb{c6}
L\left(x,\dot{x},{_a^C D^{\frac{1}{2}}_t} x,{_a^C D^{\frac{1}{2}}_t} \dot{x}\right)=\frac{A}{2}\left(\dot{x}\right)^2-\frac{1}{2}\left({_a^C D^{\frac{1}{2}}_t} \dot{x}\right)^2+\frac{1}{2}\left({_a^C D^{\frac{1}{2}}_t} x\right)^2+\int G(x)dx.
\end{equation}
It is important to stress that this Lagrangian function \eqref{c6} is different from the one in \cite{LazoCesar} and has the advantage of being a real function (the Lagrangian in  \cite{LazoCesar} is a complex valued function). In order to show that \rf{c6} give us \rf{c5}, we insert \rf{c6} into our generalized Euler-Lagrange equation \rf{t2}, obtaining
\beq
\lb{c7}
-\frac{d}{dt}\left({_t D^{\frac{1}{2}}_b}\; {_a^C D^{\frac{1}{2}}_t}\right) \dot{x}+A\ddot{x}-\left({_t D^{\frac{1}{2}}_b}\; {_a^C D^{\frac{1}{2}}_t}\right) x=G(x),
\eeq
and we follow the procedure introduced in \cite{LazoCesar}. Since (see \cite{LazoCesar})
\beq
\lb{c7b}
\lim_{\begin{array}{c}
a\rightarrow b \\
t=\frac{a+b}{2}
\end{array}}{_t D^{\frac{1}{2}}_b}\; {_a^C D^{\frac{1}{2}}_t}(\cdot)=-\frac{d}{dt}(\cdot),
\eeq
by taking the limit $a\rightarrow b$ with $t=(a+b)/2$ in \rf{c7} and using \rf{c7b} we get \rf{c5}.

An example of a real-world jerk system described by \eqref{c5} is an accelerated charged particle under a conservative force $G(x)$ and a frictional force proportional to the velocity. In this case, the first term in the Lagrangian \eqref{c6} is the kinetic energy of a particle of mass $m=-A$, the last term is minus the potential energy of the conservative force $-G(x)$, and the two terms containing fractional derivatives can be regarded as potential energies for non-conservatives forces. In particular, the third term can be interpreted as the potential energy of the frictional force $\dot{x}$, since, for $\Delta t=b-a<<1$, we get
\beq
\lb{c8}
\frac{1}{2}\left({_a^C D^{\frac{1}{2}}_t} x\right)^2\approx \frac{1}{2}\left(\frac{\Gamma(1)}{\Gamma(\frac{3}{2})}\right)(\dot{x})^2\Delta t \approx \frac{2}{\pi}\dot{x}\Delta x,
\eeq
that, apart a multiplicative constant $\frac{2}{\pi}$, coincides with the work done from the frictional force $\dot{x}$ in the displacement $\Delta x \approx \dot{x}\Delta t$. Finally, the second term in the Lagrangian \eqref{c6} can be interpreted as the potential energy for the radiation recoil force, since
\beq
\lb{c8b}
\frac{1}{2}\left({_a^C D^{\frac{1}{2}}_t} \dot{x}\right)^2\approx \frac{1}{2}\left(\frac{\Gamma(1)}{\Gamma(\frac{3}{2})}\right)(\ddot{x})^2\Delta t = \frac{2}{\pi}(\ddot{x})^2\Delta t,
\eeq
that coincides, again apart a constant $\frac{2}{\pi}$, with the energy lost in the time interval $\Delta t$ by the radiation recoil force \cite{LazoCesar} of a particle of charge $e=\pm\sqrt{\frac{3c^3}{2}}$, where $c$ is the speed of light. If we define the canonical variables
\beq
\lb{c9}
q_1=\dot{x}, \;\;\; q_{\frac{1}{2}}={_a^C D^{\frac{1}{2}}_t} x, \;\;\; q_{\frac{3}{2}}={_a^C D^{\frac{1}{2}}_t} \dot{x},
\eeq
and
\beq
\lb{c10}
p_1=\frac{\partial L}{\partial q_1}=A\dot{x}, \;\;\; q_{\frac{1}{2}}=\frac{\partial L}{\partial q_{\frac{1}{2}}}={_a^C D^{\frac{1}{2}}_t} x, \;\;\; q_{\frac{3}{2}}=\frac{\partial L}{\partial q_{\frac{3}{2}}}=-{_a^C D^{\frac{1}{2}}_t} \dot{x},
\eeq
we obtain the Hamiltonian
\beq
\lb{c11}
H=q_1p_1+q_{\frac{1}{2}}p_{\frac{1}{2}}+q_{\frac{3}{2}}p_{\frac{3}{2}}-L=\frac{A}{2}\left(\dot{x}\right)^2-\frac{1}{2}\left({_a^C D^{\frac{1}{2}}_t} \dot{x}\right)^2+\frac{1}{2}\left({_a^C D^{\frac{1}{2}}_t} x\right)^2-\int G(x)dx
\eeq
that is the sum of all energies. From Corollary \ref{cor:FOCP:CV} and the Lagrangian \eqref{c6} we obtain the following conserved quantity under time translations ($\phi_1(\varepsilon,x(t))= x(t)$ and $\phi_2(\varepsilon,t) = t+\varepsilon$)
\beq
\lb{c12}
\frac{d}{dt}\left[H+\frac{3}{2}\left({_a^C D^{\frac{1}{2}}_t} \dot{x}\right)^2-\frac{1}{2}\left({_a^C D^{\frac{1}{2}}_t} x\right)^2-A\left(\dot{x}\right)^2+2\int G(x)dx\right]=0.
\eeq
Then, the conserved quantity under time translation is given by
\beq
\lb{c13}
\frac{A}{2}\left(\dot{x}\right)^2-\left({_a^C D^{\frac{1}{2}}_t} \dot{x}\right)^2-\int G(x)dx=\mbox{constant},
\eeq
where we used \eqref{c11} in \eqref{c12}. Finally, it is evident from the condition \eqref{c12} that the total energy (given by the Hamiltonian) is not a conserved quantity, as we should expected for a system under non-conservative forces. Actually, only locally (in time) the energy is conserved. If we consider very short time intervals, by taking the limit $a\rightarrow b$, we have ${_a^CD_t}^{\frac{1}{2}}x\rightarrow 0$ and ${_a^CD_t}^{\frac{1}{2}}\dot{x}\rightarrow 0$. In this case \eqref{c11} and \eqref{c12} reduces to $H=\frac{A}{2}\left(\dot{x}\right)^2-\int G(x)dx$ and $\frac{dH}{dt} = 0$, respectively.

\subsection{Example 2: nonlinearity depending on $x$, $\dot{x}$ and $\ddot{x}$}

It is important to stress that \rf{c5} is the only chaotic jerk system containing nonlinearity depending only on $x$ \cite{jerk,jerk1,jerk2a,jerk2b,jerk2c,jerk2d,jerk2e,jerk2f,jerk2g}. For jerk systems with more complex nonlinearities, as for example $x\dot{x}$, $\dot{x}^2$ and $x\ddot{x}$, it is not possible to formulate a simple Lagrangian, depending only on $x$ and its derivatives, by using classical calculus of variation. However, by using our Euler-Lagrange equation \rf{t2} we can formulate a simple Lagrangian for these chaotic jerk systems \cite{jerk,jerk1,jerk2a,jerk2b,jerk2c,jerk2d,jerk2e,jerk2f,jerk2g}. For example:
\begin{equation}
\lb{c14}
L=\frac{A}{2}\left(\dot{x}\right)^2-\frac{1}{2}\left({_a^C D^{\frac{1}{2}}_t} \dot{x}\right)^2+\frac{1}{2}\left({_a^C D^{\frac{1}{2}}_t} x^2\right){_a^C D^{\frac{1}{2}}_t} \dot{x}-\frac{x^2}{2} \Longrightarrow \dddot{x}+A\ddot{x}-\dot{x}^2+x=0,
\end{equation}
\begin{equation}
\lb{c15}
L=\frac{A}{2}\left(\dot{x}\right)^2-\frac{1}{2}\left({_a^C D^{\frac{1}{2}}_t} \dot{x}\right)^2-\frac{1}{4}\left({_a^C D^{\frac{1}{2}}_t} x^2\right){_a^C D^{\frac{1}{2}}_t} x-\frac{x^2}{2} \Longrightarrow \dddot{x}+A\ddot{x}-x\dot{x}+x=0,
\end{equation}
\begin{equation}
\lb{c16}
L=\frac{A}{2}x\left(\dot{x}\right)^2-\frac{1}{2}\left({_a^C D^{\frac{1}{2}}_t} \dot{x}\right)^2+\frac{A+2}{4}\left({_a^C D^{\frac{1}{2}}_t} x^2\right){_a^C D^{\frac{1}{2}}_t} \dot{x}-\frac{x^2}{2} \Longrightarrow \dddot{x}+Ax\ddot{x}-\dot{x}^2+x=0.
\end{equation}
Since \eqref{c14}, \eqref{c15} and \eqref{c16} are autonomous problems, we have from Corollary \ref{cor:FOCP:CV} the following conserved quantities under time translations ($\phi_1(\varepsilon,x(t))= x(t)$ and $\phi_2(\varepsilon,t) = t+\varepsilon$):
\beq
\lb{c17}
\frac{A}{2}\left(\dot{x}\right)^2-\left({_a^C D^{\frac{1}{2}}_t} \dot{x}\right)^2+\frac{1}{2}\left({_a^C D^{\frac{1}{2}}_t} x^2\right){_a^C D^{\frac{1}{2}}_t} \dot{x}+\frac{x^2}{2}=\mbox{constant},
\eeq
\beq
\lb{c18}
\frac{A}{2}\left(\dot{x}\right)^2-\left({_a^C D^{\frac{1}{2}}_t} \dot{x}\right)^2+\frac{x^2}{2}=\mbox{constant},
\eeq
\beq
\lb{c19}
\frac{A}{2}x\left(\dot{x}\right)^2-\left({_a^C D^{\frac{1}{2}}_t} \dot{x}\right)^2+\frac{A+2}{4}\left({_a^C D^{\frac{1}{2}}_t} x^2\right){_a^C D^{\frac{1}{2}}_t} \dot{x}+\frac{x^2}{2}=\mbox{constant},
\eeq
respectively.

\subsection{Example 3: internal symmetry}

Finally, let us consider a problem displaying internal symmetry. For example, the following Lagrangian function
\beq
\lb{c20}
L\left(t,x,\dot{x},{_a^C D^{\frac{1}{2}}_t} \ln(\dot{x})\right)=\frac{1}{2}\left({_a^C D^{\frac{1}{2}}_t} \ln(\dot{x})\right)^2+t\frac{\dot{x}}{x}
\eeq
describes the jerk system
\beq
\lb{c21}
\dddot{x}-2\frac{\ddot{x}^2}{\dot{x}}-\frac{\dot{x}^2}{x}=0,
\eeq
where we used the generalized Euler-Lagrange equation \eqref{t2} and \eqref{c7b} to obtain \eqref{c21}.

The Lagrangian \eqref{c20} has a continuous symmetry related to scale changes of $x$, since the transformation $\phi_1(\varepsilon,x(t))= e^{\varepsilon c}x(t)$ ($c\in \mathbb{R}$) and $\phi_2(\varepsilon,t) = t$ leaves \eqref{c20} invariant. Thus, from the generalized Noether-type theorem \ref{theo:tnadf1} we obtain
\beq
\lb{c22}
t+\frac{x}{\dot{x}}\;{_t D^{\frac{1}{2}}_b}{_a^C D^{\frac{1}{2}}_t} \ln(\dot{x})=\mbox{constant}
\eeq
as a conserved quantity related to a scale change symmetry.


\section{Conclusions}
\label{sec:C}

In the present work, we obtained a Euler-Lagrange equation for Lagrangians depending on fractional derivatives of nonlinear functions of the unknown function $x$ and $\dot{x}$.
Our Euler-Lagrange equation generalizes a previous result from one of us \cite{LazoJerk} and it enables us to obtain Lagrangians for nonlinear open and dissipative systems, and consequently, it allows us to use the most advanced methods of classical mechanic to study these systems. Furthermore, and our main result, we formulated a Noether-type theorem for these problems, that enable us to obtain conservative quantities for nonlinear dissipative systems. All conservation laws in Physics, for examples, conservation of energy and momentum, are associated with the invariance of the action functional under continuous transformations. Within this context, the generalization of the Noether Theorem for Lagrangians depending on fractional derivatives of nonlinear functions is important to study the symmetries of nonlinear systems. 
In order to illustrate the potential of application of our results, we obtain Lagrangians for some nonlinear chaotic jerk systems, and we analyze the conservation laws related to time translations and internal symmetries.
Finally, in which concerns developments and applications of our results, there are many directions of investigation left to explore. An interesting example would be the generalization of our results to optimal control problems. This and other examples are left to future works.


\section*{Acknowledgments}

The authors are grateful to the Brazilian foundations CNPq and Capes for financial support.



\section*{Disclosure statement}

The authors declare that they have no conflict of interest.


\end{document}